\documentclass[reqno]{amsart}
\usepackage{tikz-cd}
\usepackage[all,arc, cmtip]{xy}
\usepackage{amssymb, amsthm, amsmath, amsfonts}
\definecolor{darkblue}{rgb}{0,0,0.6}
\usepackage[ocgcolorlinks,colorlinks=true, citecolor=darkblue, filecolor=darkblue,
linkcolor=darkblue, urlcolor=darkblue]{hyperref}
\usepackage[nameinlink,capitalise,noabbrev]{cleveref}
\usepackage{enumitem}
\usepackage{mathtools}

\newcommand{\rB}{\mathrm{B}}
\newcommand{\rC}{\mathrm{C}}

\newcommand{\cA}{\mathcal{A}}
\newcommand{\cB}{\mathcal{B}}
\newcommand{\cC}{\mathcal{C}}

\newcommand{\cN}{\mathcal N}
\newcommand{\bbC}{\mathbb C}
\newcommand{\bbF}{\mathbb F}
\newcommand{\bbN}{\mathbb N}
\newcommand{\bbZ}{\mathbb Z}
\newcommand{\bP}{\mathbb{P}}
\newcommand{\bQ}{\mathbb{Q}}
\newcommand{\bJ}{\mathbb{J}}
\newcommand{\bK}{\mathbb{K}}

\DeclareMathOperator{\id}{id}
\DeclareMathOperator{\short}{\mathrm{sh}}

\newcommand{\tors}[2]{\langle #1,#2 \rangle}

\newcommand\blfootnote[1]{
    \begingroup
    \renewcommand\thefootnote{}\footnote{#1}
    \addtocounter{footnote}{-1}
    \endgroup
}

\numberwithin{equation}{section}
\theoremstyle{plain}
\newtheorem{thm}[equation]{Theorem}
\newtheorem{prop}[equation]{Pro\-po\-si\-tion}
\newtheorem{cor}[equation]{Corol\-lary}
\newtheorem{lem}[equation]{Lemma}
\theoremstyle{definition}
\newtheorem{defi}[equation]{Definition}
\newtheorem{ex}[equation]{Example}
\theoremstyle{remark}
\newtheorem{rem}[equation]{Remark}

\title{On Presentations of  \texorpdfstring{$\boldsymbol{K}$}{K}-groups  by generators and relations}

\author{Bernhard K\"ock}

\begin{document}

\begin{abstract}
In Grayson's combinatorial description of higher $K$-groups, the generators are bounded acyclic binary multi-complexes of arbitrary size. Generalising work by Kasprowski, Winges and the author, we show in this paper that multi-complexes of bounded size suffice and we provide the corresponding relations. Furthermore, we report on the progress in our attempt to algebraically prove the surjectivity of Quillen's d\'evissage isomorphism for $K_1$, and we give an elementary and fairly simple example in the codomain which appears to require a more sophisticated approach. 
\end{abstract}

\maketitle


\section*{Introduction}

In his seminal paper \cite{Grayson2012}, Grayson has given a description of the higher $K$-groups of an exact category $\cN$ in terms of explicit generators and relations. The goal of this paper is to contribute to the study of this description in two different ways. 

Recall that Grayson's generators of $K_1\cN$ are so-called bounded acyclic binary complexes 
\[\begin{tikzcd} \ldots \ar[r, shift left, "d"] \ar[r, shift right, "d'"']& \bP_3 \ar[r, shift left, "d"] \ar[r, shift right, "d'"'] & \bP_2 \ar[r, shift left, "d"] \ar[r, shift right, "d'"']& \bP_1 \ar[r, shift left, "d"] \ar[r, shift right, "d'"']& \bP_0\end{tikzcd}\]
in $\cN$ meaning that both 
\[\begin{tikzcd} \ldots \ar[r, "d"] & \bP_1 \ar[r,  "d"] & \bP_0\end{tikzcd} \quad \textrm{ and } \quad  \begin{tikzcd} \ldots \ar[r, "d'"] &  \bP_1 \ar[r,  "d'"] & \bP_0\end{tikzcd}\]
are bounded acyclic complexes in $\cN$. Generators of $K_2\cN$ are bounded acyclic binary complexes of bounded acyclic binary complexes, and so on. In \cite{KW17} and \cite{KKW19}, Kasprowski, Winges and the author have shown that in fact, for any $k \ge 2$, complexes of bounded length $k$ suffice to generate $K_1\cN$ and have moreover established the corresponding relations. In the case $k=2$, these relations are even fewer than the already quite compact and beautiful relations established by Nenashev in \cite{Nenashev1998}. In \cref{thm:short-complexes}, we generalise this result to all higher $K$-groups. For example for the second $K$-group, our result says that squares of acyclic binary complexes of side length 2 suffice to generate $K_2 \cN$ and it gives the corresponding relations. 

In the second part of the paper, we consider an inclusion $\cB \subset \cA$ of abelian categories which satisfies the assumptions of Quillen's D\'evissage Theorem and ask the natural  question whether the D\'evissage Theorem can be proved algebraically, i.e., using Graysons description and in particular without resorting to homotopy theory as in [Qui73]. We mainly only study this question in the situation every object of $\cA$ is of finite length and $\cB$ is the full subcategory of semisimple objects, and we only attempt to prove that the induced map $c: K_1\cB \rightarrow K_1 \cA$ is surjective. We succeed in showing purely algebraically that the class of every binary short exact sequence 
\[ \begin{tikzcd}M' \ar[r, shift left, "i"] \ar[r, shift right, "j"'] & M\ar[r, shift left, "p"] \ar[r, shift right, "q"']& M''\end{tikzcd}\]
where the length of $M$ is at most 6 belongs to the image of $c$ except possibly in the case when the length of both $M'$ and $M''$ is 3, both $U:= i(M') \cap j(M')$ and $M/(i(M') + j(M'))$ are simple and both $i^{-1}(U) \cap j^{-1}(U) $ and $M''/(pj(M') + qi(M'))$ vanish, see \cref{cor: length M at most 6}. An example of the case excluded is the binary short exact sequence
\[\begin{tikzcd} C_2 \oplus C_4 \ar[r, shift left, "i"] \ar[r, shift right, "j"'] & C_2 \oplus C_8 \oplus C_4 \ar[r, shift left, "p"] \ar[r, shift right, "q"']& C_2 \oplus C_4 \end{tikzcd}\]
in the abelian category $\cA$ of finitely generated $\bbZ/8\bbZ$-modules where $C_n$ denotes the cyclic group of order $n$ and 
\[i(\bar{a},\bar{b}) := (\bar{a},\overline{2b},0), \quad j(\bar{a},\bar{b}) := (0, \overline{4a}, \bar{b}), \quad p(\bar{a},\bar{b},\bar{c}) := (\bar{b},\bar{c}), \quad q(\bar{a},\bar{b},\bar{c}) := (\bar{a},\bar{b})\] 
(see \cref{ex: binary short exact sequence}). It seems that this sequence has to be first related to an enlarged sequence before it can be broken down into binary sequences of vector spaces over~$\bbF_2$, which looks to be a difficult problem.

\section{Higher \texorpdfstring{$K$}{K}-groups via binary acyclic multi-complexes of bounded size}

Let $\cN$ be an exact category. Let $\cC \cN$ denote the exact category of acyclic complexes  in $\cN$ supported on a finite subset of $[0,\infty)$. (Note, to avoid bulky notation in this paper, we drop the superscript q in the notation $C^\mathrm{q} \cN$ used for this category in other papers.) 

We recall from \cite[Section 3]{Grayson2012} that a {\em binary complex} $\bP=\left(\bP., \substack{d \\ {d'}}\right)$ in~$\cN$ is a graded object~$\bP.$ in $\cN$ together with two degree~$-1$ maps $d, {d'}\colon \bP. \rightarrow \bP.$ such that both $d^2=0$ and ${d'}^2=0$. 
If $d={d'}$,  the binary complex~$\bP$ is said to be {\em diagonal}.  A {\em morphism between binary complexes} is a degree~0 map between the underlying graded objects that is a chain map with respect to both differentials. Similarly to above, let $B\cN$ denote the category of acyclic bounded binary chain complexes in $\cN$ supported in non-negative degrees. 

By iterating, we obtain the exact categories $\cC_n \cN$ and $\cB_n\cN$ of {\em acyclic $n$-fold complexes} and of {\em binary acyclic $n$-fold complexes in $\cN$}. By mixing these operations, we moreover obtain categories such as $\cB\cC\cC\cB\cN$. If top and bottom differentials of a $\bP \in \cB_n(\cN)$ in any fixed direction agree, we call $\bP$ a {\em diagonal} complex. Via ``diagonalising'', we consider a category such as $\cB\cC\cC\cB\cN$ as a (full) subcategory of $\cB \cB\cC \cB \cN$ and of $\cB \cC \cB \cB \cN$, and each of the latter categories as a subcategory of $\cB_4 \cN$. Note that restricting to the top (or bottom) differentials provides a left-inverse to every such embedding. In particular, after applying any $K$-group functor, any such embedding becomes a split injection of abelian groups.

For $\boldsymbol{k} = (k_1, \ldots, k_n) \in \bbN^n$, let $\cB^{\boldsymbol{k}}_n\cN$ denote the full subcategory of $\cB_n \cN$ consisting of complexes supported on $[0,k_1] \times \ldots \times [0,k_n]$. The notation $\cC^{\boldsymbol{k}}_n\cN$ and similar notations are defined analogously. Note that $\cC_n^{\boldsymbol{k}}\cN$ and $\cB_n^{\boldsymbol{k}} \cN$ are exact categories again. 

Modulo \cite[Proposition~1.4]{HKT2017} (which shows that using complexes supported on $(-\infty, \infty)$ rather than on $[0,\infty)$ doesn't affect the associated $K$-theory), Grayson proves in \cite[Corollary~7.2]{Grayson2012} that Quillen's $n^\mathrm{th}$ $K$-group $K_n(\cN)$ is naturally isomorphic to the quotient of $K_0(\cB_n\cN)$ modulo $K_0(\cC\cB\ldots \cB \cN) + \ldots + K_0(\cB \ldots \cB \cC \cN)$.  Throughout this paper we will use this quotient as the definition of $K_n(\cN)$. 
Similarly, we define
\[K^{\boldsymbol{k}}_n(\cN) := K_0(\cB^{\boldsymbol{k}}_n\cN)\big/ \left(K_0(\cC^{k_1}\cB^{k_2}\ldots \cB^{k_n} \cN) + \ldots + K_0(\cB^{k_1} \ldots \cB^{k_{n-1}} \cC^{k_n} \cN)\right).\]
\begin{defi}\label{def:Lnk} Let $n \ge 1$.
 \begin{enumerate}
  \item\label{it:Lnka} Let $j \in \{1, \ldots, n\}$ (viewed as one of the $n$ directions in an $n$-fold complex). A \emph{binary $j$-ladder $(\bP,\bQ,\sigma,\tau)$ in $\cB_n \cN$} consists of two complexes $\bP$ and $\bQ$ in~$\cB_n \cN$ together with two isomorphisms $\sigma$ and $\tau$ between the graded objects underlying $\bP$ and $\bQ$ such that $\sigma$ commutes with the top differentials in direction~$j$, such that $\tau$ commutes with the bottom differentials in direction~$j$ and such that $\sigma$ and $\tau$ commute with both top and bottom differentials in all other directions.

\item\label{it:Lnkb} Let $\boldsymbol{k}\in \bbN^n_+$, let $j \in \{1, \ldots, n\}$ and let $(\bP, \bQ, \sigma, \tau)$ be a binary $j$-ladder in $\cB_n^{\boldsymbol{k}} \cN$. Via the differentials in direction $j$, we view $\bP$ as a binary acyclic complex of objects $\bP_i \in \cB_{n-1} \cN $, $i\ge 0$, and similarly for $\bQ$. For each $i \ge 0$,  we then obtain the binary acyclic complex
\[\begin{tikzcd}
  \bP_i\ar[r, shift left, "\sigma_i"]\ar[r, shift right, "\tau_i"'] & \bQ_i.
  \end{tikzcd}\]
of complexes in $\cB_{n-1} \cN$ and supported on $[0,1]$; we view this complex as a complex in $\cB^{\boldsymbol{k}}_n \cN$ using the displayed binary differential $(\sigma_i, \tau_i)$ as the differential in direction $j$ (note that $k_j \ge 1$). We write $\tors{\sigma_i}{\tau_i}_j$ for this complex and also for its class in $K_0(\cB^{\boldsymbol{k}}_n \cN)$ or $K_0(\cB_n\cN)$ or in any of their quotients such as $K_n^{\boldsymbol{k}}(\cN)$ or $K_n(\cN)$.
\item\label{it:Lnkc} Let $\boldsymbol{k}\in \bbN^n_+$. For $j \in \{1, \ldots, n\}$, let $R_j(\cB^{\boldsymbol{k}}_n \cN)$ denote the subgroup of $K_0(\cB^{\boldsymbol{k}}_n \cN)$ generated by the elements
 \[\bQ - \bP - \sum_{i=0}^{k_j} (-1)^i \tors{\sigma_i}{\tau_i}_j \]
 associated with binary $j$-ladders $(\bP,\bQ,\sigma,\tau)$ in $\cB_n^{\boldsymbol{k}}\cN$ where $\bP = \bQ$ (as graded objects) and where $\sigma$ and $\tau$ are involutions. We write $R_j(\cB^{\boldsymbol{k}}_n \cN)$ also for its image in $K^{\boldsymbol{k}}_n (\cN)$.  Define $L^{\boldsymbol{k}}_n(\cN)$ to be the quotient of $K_n^{\boldsymbol{k}}(\cN)$ modulo the subgroup $R_1(\cB^{\boldsymbol{k}}_n \cN)+ \ldots + R_n(\cB^{\boldsymbol{k}}_n \cN)$.
 \end{enumerate}
\end{defi}

An argument which is almost verbatim the same as the proof of \cite[Lemma~3.3]{KKW19} shows that the elements associated in \cref{def:Lnk}(3) with $j$-ladders in $\cB_n^{\boldsymbol{k}}\cN$ vanish in $K_n^{\boldsymbol{\tilde{k}}}(\cN)$ where $\boldsymbol{\tilde{k}}$ is obtained from $\boldsymbol{k}$ by adding $1$ to its $j$th component. Hence, the canonical map $K^{\boldsymbol{k}}_n (\cN) \to K_n(\cN)$ factorises via $L^{\boldsymbol{k}}_n(\cN)$.

\begin{thm}\label{thm:short-complexes}
 For every $n \ge 1$ and every $\boldsymbol{k} \in [2, \infty)^n$, the canonical map
 \[ L_n^{\boldsymbol{k}}(\cN) \to K_n(\cN) \]
 is an isomorphism.
\end{thm}

\begin{proof}
We proceed by induction on $n$. 
The case $n=1$ has been proved in \cite[Theorem~2.4]{KKW19}. Henceforth, let $n\ge 2$ and assume that the statement is true for $n-1$. Furthermore, we fix $\boldsymbol{k} \in [2, \infty)^n$ and write $\boldsymbol{k'}$ for $(k_1, \ldots, k_{n-1})$.

By \cite[Theorem~2.4]{KKW19}, the canonical map $K_1^{k_n}(\cB^{\boldsymbol{k'}}_{n-1} \cN) \to K_1 (\cB^{\boldsymbol{k'}}_{n-1}\cN )$ induces an isomorphism $L_1^{k_n} (\cB^{\boldsymbol{k'}}_{n-1} \cN) \to K_1 (\cB^{\boldsymbol{k'}}_{n-1} \cN)$. Similarly to \cref{def:Lnk}\eqref{it:Lnkc}, from every binary $j$-ladder $(\bP,\bQ, \sigma, \tau)$ in $\cB^{(\boldsymbol{k'}, \infty)}_n \cN$ for any $j \in \{1, \ldots, n-1\}$ we obtain the ladder relation
\[x:= \bQ - \bP - \sum_{i \ge 0} (-1)^i \langle \sigma_i, \tau_i\rangle_j \quad \textrm{ in }\quad K_1 (\cB^{\boldsymbol{k'}}_{n-1} \cN).\]
As we will see later, the main step required to prove \cref{thm:short-complexes} is to show that the preimage $y$ of $x$ under the isomorphism $L_1^{k_n} (\cB^{\boldsymbol{k'}}_{n-1} \cN) \to K_1 (\cB^{\boldsymbol{k'}}_{n-1} \cN)$ is again a ladder relation, more precisely, an integral linear combination of ladder relations in $L_1^{k_n} (\cB^{\boldsymbol{k'}}_{n-1} \cN)$ corresponding to binary ladders in $B_n^{\boldsymbol{k}} \cN$. 
By the proof of \cite[Theorem~2.4]{KKW19}, all homomorphisms in the canonical inductive system
\[L_1^2 (\cB^{\boldsymbol{k'}}_{n-1} \cN) \to L_1^3 (\cB^{\boldsymbol{k'}}_{n-1} \cN) \to \ldots \]
are bijective and the canonical homomorphism from the limit of this inductive system to $K_1 (\cB^{\boldsymbol{k'}}_{n-1} \cN)$ is bijective as well. Now, the $j$-ladder $(\bP, \bQ, \sigma, \tau)$ is a $j$-ladder in $\cB_n^{(\boldsymbol{k'}, l_n)} \cN$ for some $l_n \ge 1$. The proof of the main step proceeds by downward induction on $l_n$. Over the next page or so, we (only) prove the inductive step $k_n +1 \rightarrow k_n$; all other inductive steps can be similarly proved.

Let the $n$-directional Grayson shortening $\short_n(\bP)  \in \cB^{\boldsymbol{k}}_n \cN$ be defined as follows. Via the differentials in direction $n$, we view $\bP$ as a binary acyclic complex of objects $\bP_{i} \in \cB^{\boldsymbol{k'}}_{n-1} \cN$:
 \[\begin{tikzcd}
  \ldots\ar[r, shift left, "d_3"]\ar[r, shift right, "d'_3"'] & \bP_2\ar[r, shift left, "d_2"]\ar[r, shift right, "d'_2"'] & \bP_1 \ar[r, shift left, "d_1"]\ar[r, shift right, "d'_1"']& \bP_0
 \end{tikzcd}\]
Let $\bJ, \bK \in \cB^{\boldsymbol{k'}}_{n-1} \cN$ denote the kernel of $d_1$ and $d'_1$, respectively. Then $\short_n(\bP) \in \cB^{\boldsymbol{k}}_n \cN = \cB^{k_n} \cB^{\boldsymbol{k'}}_{n-1} \cN$ is defined to be the binary acyclic complex of objects in $\cB^{\boldsymbol{k'}}_{n-1} \cN$ whose top and bottom differentials are given by the following diagrams:
\[\begin{tikzcd}[row sep= tiny]
 \ldots \ar[r] & \bP_3\ar[r, "d_3"]& \bP_2\ar[r, "d_2"]\ar[d, phantom, "\oplus"] & \bJ\ar[d, phantom, "\oplus"] \\
 & \oplus & \bK\ar[r, "\id"]\ar[d, phantom, "\oplus"] & \bK \\
 & \bJ\ar[r, "\id"]\ar[d, phantom, "\oplus"] & \bJ\ar[d, phantom, "\oplus"] & \oplus \\
 & \bK\ar[r, rightarrowtail] & \bP_1\ar[r, "d_1'"] & \bP_0 \\
 \end{tikzcd} \qquad
  \begin{tikzcd}[row sep= tiny]
 \ldots \ar[r] & \bP_3\ar[r, "d_3'"] & \bP_2\ar[r, "d_2'"]\ar[d, phantom, "\oplus"] & \bK\ar[d, phantom, "\oplus"] \\
 & \oplus & \bJ\ar[r, "\id"]\ar[d, phantom, "\oplus"] & \bJ \\
 & \bK\ar[r, "\id"]\ar[d, phantom, "\oplus"] & \bK\ar[d, phantom, "\oplus"] & \oplus \\
 & \bJ\ar[r, rightarrowtail] & \bP_1\ar[r, "d_1"] & \bP_0 \\
 \end{tikzcd}\]
(Strictly speaking, the direct sum of $\bJ$ and $\bK$ in each column of the right-hand diagram needs to be transposed, but we don't do so in order to avoid depicting arrows that cross over.) Slightly deviating from the notation in \cite{KKW19}, let $s_\bJ$ denote the switching automorphism of $\bJ \oplus \bJ$ and let $s_{\bP,n} := \langle \id_{\bJ \oplus \bJ}, s_\bJ\rangle_n \in \cB^{(\boldsymbol{k'},1)}_n \cN$. (As an aside, if we replace $\bJ$ with $\bK$ here, the class of $s_{\bP, n}$ in $K_0(\cB^{(\boldsymbol{k'},1)}_n \cN)$ stays the same because $\bJ = \bK$ in $K_0 (\cB^{\boldsymbol{k'}}_{n-1} \cN)$.) The notations $\short_n(\bP)$ and $s_{\bP,n}$ are of course also defined when $\bP$ is replaced with $\bQ$ or $\langle \sigma_i, \tau_i \rangle_j$. By (the proof of) \cite[Proposition~3.9]{KKW19}, we have
\[ y = \left(- \short_n(\bQ) - s_{\bQ,n}\right) + \left(\short_n(\bP) + s_{\bP,n}\right)+ \sum_{i \ge 0} (-1)^i \left(\short_n(\langle\sigma_i, \tau_i\rangle_j) + s_{\langle \sigma_i, \tau_j \rangle_j, n}\right)\]
in $L_1^{k_n}(\cB^{\boldsymbol{k'}}_{n-1} \cN)$. By restricting $\sigma_1$ and $\tau_1$ we obtain isomorphisms $\sigma_\bJ \colon \bJ_\bP \xrightarrow{\sim} \bJ_\bQ$, $\sigma_\bK \colon \bK_\bP \xrightarrow{\sim} \bK_\bQ$, $\tau_\bJ \colon \bJ_\bP \xrightarrow{\sim} \bJ_\bQ$ and $\tau_\bK \colon \bK_\bP \xrightarrow{\sim} \bK_\bQ$, where we now more precisely write $\bJ_\bP$ for $\bJ$ and where $\bJ_\bQ$, $\bK_\bP$ and $\bK_\bQ$ are defined similarly. Let the isomorphism $\short_n(\sigma) \colon \short_n(\bP) \to \short_n(\bQ)$ of graded objects be defined by
\begin{align*}
  &\short_n(\sigma)_0 = \sigma_\bJ \oplus \sigma_\bK \oplus \sigma_0, \quad
  \short_n(\sigma)_1 = \sigma_2 \oplus \sigma_\bK \oplus \sigma_\bJ \oplus \sigma_1, \quad
  \short_n(\sigma)_2 = \sigma_3 \oplus \sigma_\bJ \oplus \sigma_\bK
\end{align*}
and $\short_n(\sigma)_i = \sigma_{i+1}$ for all $i \ge 3$. Similarly, the isomorphism $\short_n(\tau) \colon \short_n(\bP) \to \short_n(\bQ)$ is defined by
\begin{align*}
  &\short_n(\tau)_0 = \tau_\bJ \oplus \tau_\bK \oplus \tau_0, \quad
  \short_n(\tau)_1 = \tau_2 \oplus \tau_\bK \oplus \tau_\bJ \oplus \tau_1, \quad
  \short_n(\tau)_2 = \tau_3 \oplus \tau_\bJ \oplus \tau_\bK
  \end{align*}
and $\short_n(\tau)_i = \tau_{i+1}$ for $i \geq 3$. Then $(\short_n(\bP), \short_n(\bQ), \short_n(\sigma), \short_n(\tau))$
is a binary $j$-ladder in $\cB^{\boldsymbol{k}}_n\cN$. As restricting to the $i^\mathrm{th}$ component in direction $j$ commutes with forming $\bJ$ and $\bK$, we have
\[\langle\short_n(\sigma)_i, \short_n(\tau)_i \rangle_j = \short_n(\langle \sigma_i, \tau_i \rangle_j)\]
for all $i \ge 0$. Furthermore, from the binary $j$-ladder $(\bP, \bQ, \sigma, \tau)$ we obtain the binary $j$-ladder $(\bJ_\bP, \bJ_\bQ, \sigma_\bJ, \tau_\bK)$ in $\cB^{\boldsymbol{k'}}_{n-1}\cN$, which in turn induces the binary $j$-ladder $(s_{\bP,n}, s_{\bQ,n}, (\sigma_\bJ \oplus \sigma_\bJ)^{ 2}, (\tau_\bK \oplus \tau_\bK)^{ 2})$ in $\cB^{\boldsymbol{k}}_n \cN$. Here, $(\sigma_\bJ \oplus \sigma_\bJ)^{ 2}$ denotes the isomorphism from $s_{\bP,n}$ to $s_{\bQ,n}$ which is given by $\sigma_\bJ \oplus \sigma_\bJ$ in $n$-degrees $0$ and $1$; similarly for $(\tau_\bK \oplus \tau_\bK)^{ 2}$. Again, as restricting to the $i^\mathrm{th}$ component in direction $j$ commutes with forming $\bJ$ and $\bK$, we have
\[\langle ((\sigma_\bJ \oplus \sigma_\bJ)^{ 2})_i, ((\tau_\bK \oplus \tau_\bK)^{ 2})_i \rangle_j = s_{\langle \sigma_i, \tau_i \rangle_j, n}\]
for all $i \ge 0$. Hence we have shown that $y$ is the sum of the relations given by the binary $j$-ladders
\[(\short_n(\bP), \short_n(\bQ), \short_n(\sigma), \short_n(\tau)) \quad \textrm{ and } \quad (s_{\bP,n}, s_{\bQ,n}, (\sigma_\bJ \oplus \sigma_\bJ)^{ 2}, (\tau_\bK \oplus \tau_\bK)^{ 2})\]
in $\cB^{\boldsymbol{k}}_n \cN$, as claimed above. 

In particular, we obtain the isomorphism
\begin{eqnarray}
\lefteqn{K_1^{k_n}(\cB^{\boldsymbol{k'}}_{n-1}\cN)\big/\left(R_1(\cB^{\boldsymbol{k}}_n\cN) + \ldots + R_n(\cB^{\boldsymbol{k}}_n \cN)\right)} \nonumber \\
&\cong& K_1(\cB_{n-1}^{\boldsymbol{k'}}\cN)\Big/ \left(R_1(\cB^{(\boldsymbol{k'},\infty)}_n \cN) + \ldots + R_{n-1}(\cB^{(\boldsymbol{k'},\infty)}_n \cN)\right). \label{equ: iso}
\end{eqnarray}
We therefore conclude
\begin{eqnarray*}
\lefteqn{L_n^{\boldsymbol{k}}(\cN)}\\
&\cong& K_0(\cB^{\boldsymbol{k}}_n \cN) \Big/ \Big(K_0(\cC^{k_1} \cB^{k_2} \ldots \cB^{k_n} \cN) + \ldots + K_0(\cB^{k_1} \ldots \cB^{k_{n-1}}\cC^{k_n}\cN)\\
&& \hspace*{18em}  + R_1 (\cB^{\boldsymbol{k}}_n \cN) + \ldots + R_n(\cB^{\boldsymbol{k}}_n\cN)\Big)\\
&\cong& K_1^{k_n}(\cB^{\boldsymbol{k'}}_{n-1} \cN) \Big/ \Big((K_1^{k_n}(\cC^{k_1}\cB^{k_2} \ldots \cB^{k_{n-1}} \cN) + \ldots \\
&& \hspace*{5em} + K_1^{k_n}(\cB^{k_1} \ldots \cB^{k_{n-2}} \cC^{k_{n-1}} \cN)
+ R_1 (\cB^{\boldsymbol{k}}_n \cN) + \ldots + R_n(\cB^{\boldsymbol{k}}_n\cN)\Big)\\
&\cong& K_1(\cB^{\boldsymbol{k'}}_{n-1} \cN) \Big/ \Big( (K_1(\cC^{k_1}\cB^{k_2} \ldots \cB^{k_{n-1}} \cN)+ \ldots + K_1(\cB^{k_1} \ldots \cB^{k_{n-2}} \cC^{k_{n-1}} \cN)\\
&& \hspace*{0.5em} \left.+ R_1(\cB^{(\boldsymbol{k'},\infty)}_n \cN)+ \ldots + R_{n-1}(\cB^{(\boldsymbol{k'},\infty)}_n \cN)\right) \qquad \textrm{(by isomorphism (\ref{equ: iso}))}\\
&\cong& K_0(\cB^{(\boldsymbol{k'}, \infty)}_n\cN)\Big/ \left(K_0(\cC^{k_1} \cB^{k_2} \ldots \cB^{k_{n-1}}\cB \cN)+\ldots + K_0(\cB^{k_1} \ldots \cB^{k_{n-1}} \cC \cN) \right. \\
&&\hspace*{14em} + \left.R_1(\cB^{(\boldsymbol{k'},\infty)}_n \cN)+ \ldots + R_{n-1}(\cB^{(\boldsymbol{k'},\infty)}_n \cN)\right)\\
&\cong& K^{\boldsymbol{k'}}_{n-1}(\cB \cN) \Big/ \left(K^{\boldsymbol{k'}}_{n-1}(\cC \cN) + R_1(\cB^{(\boldsymbol{k'},\infty)}_n \cN)+ \ldots + R_{n-1}(\cB^{(\boldsymbol{k'},\infty)}_n \cN)\right)\\
&\cong& L^{\boldsymbol{k'}}_{n-1}(\cB \cN) \big/ L^{\boldsymbol{k'}}_{n-1}(\cC \cN) \\
&\cong& K_{n-1}(\cB \cN)/ K_{n-1}(\cC\cN) \hspace{7em} \textrm{(by the inductive hypothesis)}\\
&\cong& K_0(\cB_n \cN)\big/\left(K_0(\cC \cB \ldots \cB \cN) + \ldots + K_0(\cB \ldots \cB \cC \cN)\right)\\
&\cong& K_n(\cN),
\end{eqnarray*}
as was to be proved.
\end{proof}

\section{On the D\'evissage Theorem for \texorpdfstring{$K_1$}{K1}}

Throughout this section, $\cA$ is a small abelian category and $\cB$ is a non-empty full subcategory of $\cA$ that is closed under taking subobjects, quotient objects and finite products in~$\cA$. Then $\cB$ is an abelian category as well and the inclusion $\cB \subseteq \cA$ is exact. We furthermore assume that every object $M$ of~$\cA$ has a finite filtration
\[0 = M_0 \subset M_1 \subset \ldots \subset M_k = M\]
such that the quotient $M_i/M_{i-1}$ is in $\cB$ for each $i$. Then Quillen's D\'evissage Theorem \cite[Theorem~4]{Quil1} is the following statement.

\begin{thm}\label{thm:devissage}
For every $n \ge 0$, the induced homomorphism $K_n(\cB) \rightarrow K_n(\cA)$ is bijective.
\end{thm}

It applies to the following two main standard situations.\\

\noindent (I) $\cA$ is an abelian category such that every object in $\cA$ is of finite length, and $\cB$ is the full subcategory of $\cA$ consisting of semi-simple objects.\\

\noindent (II) $\cA$ is the abelian category of finitely generated $R$-modules and $\cB$ is the full subcategory of $R/I$-modules where $R$ is a Noetherian commutative ring and $I$ is a nilpotent ideal in $R$.\\

In both of these situations, the inclusion $\cB \subseteq \cA$ satisfies the following condition:\\

\noindent $(*)$ For any $M \in \cA$, the set of sub-objects of $M$ which belong to $\cB$ contains a maximal one. \\

In situation (I), the set of all sub-objects of any object in $\cA$ is in fact Noetherian. In situation (II), every sub-module of any $M \in \cA$ which belongs to $\cB$ is contained in the sub-module $\{x \in M : Ix =0\} \in \cB$. In general, any maximal element in the set $\mathcal{S}$ of sub-objects of any $M \in \cA$  which belong to $\cB$ is then in fact the greatest element of $\mathcal{S}$ because the sum of any two sub-objects which belong to $\cB$ again belongs to $\cB$.\\

Following Grayson's description \cite{Grayson2012} of higher $K$-groups in terms of generators and relations, it is natural to ask whether the D\'evissage Theorem can be proved algebraically, in particular without resorting to homotopy theory as in \cite{Quil1}. Here, we (only) consider the problem of algebraically proving the surjectivity of the induced map \[c \colon K_1(\cB) \to K_1(\cA),\] and mainly only in the situation (I). We will work with Nenashev's presentation of~$K_1$:

\begin{defi}\label{def:NenashevK1}
{\em Let $\cN$ be an exact category. Then Nenashev's $K_1$-group $K_1(\cN)$ of~$\cN$} is defined as the abelian group generated by binary acylic complexes $\bP$ of length $2$ subject to the following relations:
\begin{enumerate}
 \item\label{it:NenashevK1a} If $\bP$ is a diagonal complex, then $\bP=0$.
 \item\label{it:NenashevK1b} If
 \[\begin{tikzcd}
   P'_2\ar[d, shift right]\ar[d, shift left]\ar[r, shift left]\ar[r, shift right] & P'_1\ar[d, shift right]\ar[d, shift left]\ar[r, shift left]\ar[r, shift right]& P'_0\ar[d, shift right]\ar[d, shift left] \\
  P_2\ar[d, shift right]\ar[d, shift left]\ar[r, shift left]\ar[r, shift right] & P_1\ar[d, shift right]\ar[d, shift left]\ar[r, shift left]\ar[r, shift right]& P_0\ar[d, shift right]\ar[d, shift left]\\
  P''_2\ar[r, shift left]\ar[r, shift right] & P''_1\ar[r, shift left]\ar[r, shift right]& P''_0
 \end{tikzcd}\]
is a diagram in $\cN$ such that all rows and columns are binary acyclic complexes, top differentials commute with top differentials and bottom differentials commute with bottom differentials, then
\[\bP_0 - \bP_1 +\bP_2 = \bP' - \bP + \bP''.\]
 \end{enumerate}
\end{defi}

Nenashev proves in \cite{Nenashev1998} that $K_1(\cN)$ is canonically isomorphic  to Quillen's $K_1$-group of $\cN$. This justifies the notation $K_1(\cN)$. Furthermore, it has been purely algebraically proved in \cite{KW17} and \cite{KKW19} that $K_1(\cN)$ is isomorphic to Grayson's $K_1$-group of $\cN$. 

Compared to Grayson's presentation, Nenashev's presentation has the advantage that we need to check only that the classes of binary acyclic complexes of length $2$ rather than of arbitrary finite length belong to the image of $c$. On the other hand, compared to the presentation of $K_1$ given in \cite{KKW19}, Nenashev's presentation has the advantage that, from the outset, we have more relations available to manipulate elements in $K_1$.

Henceforth, let
\[ \begin{tikzcd}M'\ar[r, shift left, "i"]\ar[r, shift right, "j"'] & M \ar[r, shift left, "p"]\ar[r, shift right, "q"']& M''\end{tikzcd}\]
be a binary short acyclic sequence in $\cA$ and let $x$ denote its class in $K_1(\cA)$ (considered as a complex supported on $[0,2]$).

\begin{lem}\label{lem:binary isomorphism}
If $\cB \subseteq \cA$ satisfies the condition $(*)$ above, the following statement holds: if $M'=0$ or $M''=0$, then $x$ belongs to the image of $c$.
\end{lem}

\begin{proof}
We first consider the case that $M'=0$.
Then $ \begin{tikzcd} M \ar[r, shift left, "p"]\ar[r, shift right, "q"']& M''\end{tikzcd}$ is isomorphic to $\begin{tikzcd} M \ar[r, shift left, "\id"]\ar[r, shift right, "\alpha"']& M\end{tikzcd}$ where $\alpha := p^{-1} \circ q$. We may therefore assume that $M=M'' \not=0$ and $p=\id$.
Let $N$ denote the maximal sub-object of $M$ that belongs to $\cB$.   Then,  $q|_N$ is a well-defined automorphism of $N$. We therefore obtain the following commutative diagram with acylic rows and vertical binary isomorphisms:
\[\begin{tikzcd}
   N\ar[d, shift right, "q|_N"' ]\ar[d, shift left, "\id"]\ar[r] & M \ar[d, shift right, "q"']\ar[d, shift left, "\id"]\ar[r]& M/N\ar[d, shift right,"\bar{q}"']\ar[d, shift left, "\id"] \\
  N \ar[r] & M \ar[r]& M/N
 \end{tikzcd}\]
By Nenashev's relations, $x$ is equal to the sum of the classes of the left-hand column and of the right-hand column. By construction, the left-hand column belongs to~$\cB$. Projecting any filtration of $M$ which quotients in $\cB$ onto $M/N$ gives a filtration of $M/N$ whose quotients are again in $\cB$ and which is shorter than the given filtration of $M$. Hence, by induction, we may assume that the class of the right-hand column belongs to the image of $c$. Then so does $x$, as claimed.

Finally, the case $M''=0$ follows from the case $M'=0$ by the shifting lemma \cite[Lemma~1.6]{HKT2017} or by arguments similar to the those above.
\end{proof}

Contrary to the previous lemma, the following example doesn't require the assumption that $\cB \subseteq \cA$ satisfies $(*)$.

\begin{ex}\label{ex: permutation}
Let $N \in \cA$ and $n \ge 0$. Then, for any permutation $\sigma$ of $\{1, \ldots, n\}$, the class of $\begin{tikzcd} N^n \ar[r, shift left, "\id"] \ar[r, shift right, "\sigma"']& N^n \end{tikzcd}$ belongs to the image of~$c$. Indeed, if
\[0 = N_0 \subset N_1 \subset \ldots \subset N_k = N\]
is a filtration of $N$ with quotients in $\cB$, then $\begin{tikzcd} N_i^n \ar[r, shift left, "\id"] \ar[r, shift right, "\sigma"']& N_i^n \end{tikzcd}$, $i = 0, \ldots, k$, form a well-defined filtration of $\begin{tikzcd} N^n \ar[r, shift left, "\id"] \ar[r, shift right, "\sigma"']& N^n \end{tikzcd}$ such that the objects in the quotients belong to $\cB$.
\end{ex}

\begin{lem}\label{lem: split}
Suppose the statement in \cref{lem:binary isomorphism} is true. If both top and bottom sequence in the binary acyclic sequence
\[ \begin{tikzcd}M'\ar[r, shift left, "i"]\ar[r, shift right, "j"'] & M \ar[r, shift left, "p"]\ar[r, shift right, "q"']& M''\end{tikzcd}\]
split, then its class $x$ belongs to the image of $c$.
\end{lem}

\begin{proof}
We choose splittings $s: M \to M'$ and $t: M \to M'$ of $i$ and $j$, respectively. We then have the following commutative diagram with acyclic rows and vertical binary isomorphism:
\[\begin{tikzcd}
   M'\ar[d, equal] \ar[r, shift left, "i"] \ar[r, shift right, "j"'] & M \ar[d, shift right, "\binom{t}{q}"']\ar[d, shift left, "\binom{s}{p}"]\ar[r, shift left, "p"] \ar[r, shift right, "q"']& M''\ar[d, equal]\\
  M'\ar[r, "\binom{\id}{0}"] & M' \oplus M'' \ar[r, "(0 \, \id)"]& M''
 \end{tikzcd}\]
By Nenashev's relations, $x$ is therefore equal to the class of $\begin{tikzcd} M \ar[r, shift left, "\binom{s}{p}"] \ar[r, shift right, "\binom{t}{q}"'] & M' \oplus M''\end{tikzcd}$, which belongs to the image of $c$ by assumption.
\end{proof}

\begin{rem}
The proof above and \cref{ex: permutation} show that, without assuming condition $(*)$, $x$ belongs to the image of $c$ also in the following situation:
$M= N^n$, $M' = N^k$, $M'' = N^{n-k}$, both $i$ and $j$ are given by inclusions of $\{1, \ldots, k\}$ in $\{1, \ldots, n\}$ and $p$ and $q$ are the corresponding complementary projections (where $N$ is any object in~$\cA$ and $n \ge k \ge 0$).
\end{rem}

\begin{lem}\label{lem: equal image} \mbox{}\\ 
(a) Suppose there is a subobject $U'$ of $M'$ such that $i(U') = j(U')=: U$. Then $x$ is equal to the sum of the classes of
\[\begin{tikzcd}   U' \ar[r, shift left, "i|_U"]\ar[r, shift right, "j|_U"'] & U\ar[r]& 0 \end{tikzcd} \quad \textrm{ and } \quad
\begin{tikzcd}  M'/U' \ar[r, shift left, "\bar{i}"] \ar[r, shift right, "\bar{j}"'] & M/U \ar[r, shift left, "\bar{p}"]\ar[r, shift right, "\bar{q}"']& M'' \end{tikzcd}.\]
(b) Suppose there is a subobject $U''$ of $M''$ such that $p^{-1}(U'') = q^{-1}(U'')=:U$. Then $x$ is equal to the sum of the classes of
\[\begin{tikzcd}  M' \ar[r, shift left, "i"] \ar[r, shift right, "j"'] & U \ar[r, shift left, "p|_U"]\ar[r, shift right, "q|_U"']& U'' \end{tikzcd} \quad \textrm{ and } \quad
\begin{tikzcd}   0 \ar[r] & M/U  \ar[r, shift left, "\bar{p}"]\ar[r, shift right, "\bar{q}"'] & M''/U'' \end{tikzcd}.\]
\end{lem}

\begin{proof}
This is obvious.
\end{proof}

Given \cref{lem:binary isomorphism}, the previous lemma is of immediate use in the extreme case $i(M') = j(M')$. The next lemma deals with the opposite extreme case, i.e., when $i(M') \cap j(M') =0$.

\begin{lem}\label{lem: images direct}\mbox{}\\
(a) Suppose that
$\ker ( M' \oplus M' \xrightarrow{(i\; j)} M) =0$. Then we obtain an induced short binary acyclic sequence of the form
\[\begin{tikzcd} M' \ar[r, shift left] \ar[r, shift right] & M'' \ar[r, shift left] \ar[r, shift right] & M/ \left(i(M') + j(M')\right)\end{tikzcd}. \]
If the class of this induced sequence belongs to the image of $c$, then so does $x$. \\
(b) Suppose that ${coker} (M \xrightarrow{\binom{p}{q}} M'' \oplus M'' )= 0$. Then we obtain an induced short binary acyclic sequence of the form
\[\begin{tikzcd} i(M') \cap j(M') \ar[r, shift left] \ar[r, shift right] & M' \ar[r, shift left] \ar[r, shift right] & M''\end{tikzcd}. \]
If the class of this induced sequence belongs to the image of $c$, then so does $x$.
\end{lem}

Note that the middle object in these induced sequences is $M''$ and $M'$, respectively, and hence usually `smaller' than the middle object $M$ of the given sequence.

\begin{proof}
In case (a) we obtain the commutative diagram
 \[\begin{tikzcd}
   M'\ar[d, equal]\ar[r, shift left, "\binom{\id}{0}"]\ar[r, shift right, "\binom{0}{\id}"'] & M'\oplus M' \ar[d, "(i\;j)"]\ar[r, shift left, "(0 \id )"]\ar[r, shift right, "(\id 0)"']& M'\ar[d, shift right]\ar[d, shift left] \\
  M'\ar[d]\ar[r, shift left, "i"]\ar[r, shift right, "j"'] & M\ar[d] \ar[r, shift left, "p"]\ar[r, shift right, "q"']& M''\ar[d, shift right]\ar[d, shift left]\\
  0 \ar[r] & M/(i(M') +j(M')) \ar[r, equal]& M/(i(M')+j(M'))
 \end{tikzcd}\]
with acyclic rows and columns where the morphisms in the right-hand column are induced by the other morphisms. By Nenashev's relations, $x$ is equal to the sum of the classes of the top row and of the right-hand column. The top row belongs to the image of $c$ by (the proof of) \cref{lem: split} (and \cref{ex: permutation}).
By assumption, the class of the right-hand column belongs to the image of $c$ as well. Then so does~$x$, as claimed.

Case (b) is dual to case (a) and hence follows from case (a).
\end{proof}

\begin{prop}\label{prop: length at most 3}
Suppose that the inclusion $\cB \subseteq \cA$ is as in (I). Then the element~$x$ belongs to the image of $c$ in any of the following two cases:\\
(A) $\mathrm{length}(M') \le 3$ except possibly if
\begin{itemize}[leftmargin=2.5em]
\item  $\mathrm{length}(M')=3$, $U:= i(M') \cap j(M')$ is simple and $i^{-1} (U) \cap j^{-1}(U) =0$

or

\item $\mathrm{length}(M') =3$ and $i(M') \cap j(M') =0$.
\end{itemize}
(B) $\mathrm{length}(M'') \le 3$ except possibly if
\begin{itemize}[leftmargin=2.5em]
\item $\mathrm{length}(M'')=3$, $M/ \left(i(M') + j(M') \right)$ is simple, $M''/ \left( pj(M') + qi(M')\right) =0$

or

\item $\mathrm{length}(M'')=3$ and $M/(i(M') + j(M')) =0$.
\end{itemize}
\end{prop}

\begin{proof} Note $i(M') \cap j(M') = \ker \left(\binom{p}{q} \colon M \rightarrow M'' \oplus M'' \right)$ and $M/ (i(M') + j(M')) = \mathrm{coker} ((i, j) \colon M'\oplus M' \rightarrow M )$. Further, $i^{-1}(U) \cap j^{-1} (U)$ and $M''/(pj(M') + qi(M'))$ are the kernel and cokernel of
\[\binom{pj}{qi} \colon M' \to M'' \oplus M'' \qquad \textrm{ and } \qquad (pj, qi): M' \oplus M' \to M'',\]
respectively. Therefore, part (B) is dual to part (A), and it suffices to prove part~(A).
\begin{enumerate}[left= 0pt]
\item The case $\mathrm{length}(M') =0$ is covered by \cref{lem:binary isomorphism}.
\item If $\mathrm{length}(M') = 1$, we either have $i(M') = j(M')$ or $i(M') \cap j(M') =0$.
\begin{enumerate}[leftmargin= 1.8em]
\item The case $i(M') = j(M')$ follows from \cref{lem: equal image} and \cref{lem:binary isomorphism}.
\item The case $i(M') \cap j(M') =0$ follows from \cref{lem: images direct} by induction on $\mathrm{length}(M)$.
\end{enumerate}
\item If $\mathrm{length}(M') =2$, then $U:= i(M') \cap j(M')$ is of length 2, 1 or 0.
\begin{enumerate}[leftmargin= 1.8em]
\item If it is 2, we again apply \cref{lem: equal image} and \cref{lem:binary isomorphism}.
\item If it is 1, then $U':= i^{-1}(U) \cap j^{-1}(U)$ is of length 1 or 0.
\begin{enumerate}[leftmargin=1.8em]
\item The case $\mathrm{length}(U')=1 $ follows from \cref{lem: equal image}, \cref{lem:binary isomorphism} and the case $\mathrm{length}(M') =1$ considered above.
\item The case $U'=0$ is dealt with as follows. We first note that $M' = j^{-1}(U) \oplus i^{-1} (U)$ as both sides are of length 2. This implies that\\

$i(M') + j(M') = ij^{-1}(U) \oplus U \oplus  ji^{-1}(U)$.\\

Note the sum on the right-hand side is indeed direct because otherwise we would have $ij^{-1}(U) \cap j(M')\not= 0$ which in turn would mean that $ij^{-1}(U) \subset j(M')$ and then that $i(M') = ij^{-1}(U) \oplus U \subset j(M')$ and finally that $U = i(M') \cap j(M') = i(M')$ would be of length 2. We therefore obtain the following  commutative diagram with acyclic rows and columns:\\

$\begin{tikzcd}
   U^{ 2}\ar[d, "(j^{-1} \, i^{-1})"]\ar[r, shift left, "\varphi"] \ar[r, shift right, "\psi"'] & U^3 \ar[d, "(ij^{-1} \id ji^{-1})"] \ar[r, shift left, "(0 \, 0 \id)"]\ar[r, shift right, "(\id 0\, 0)"']& U\ar[d, shift right]\ar[d, shift left] \\
  M'\ar[d]\ar[r, shift left, "i"]\ar[r, shift right, "j"'] & M \ar[d] \ar[r, shift left, "p"]\ar[r, shift right, "q"']& M''\ar[d, shift right]\ar[d, shift left]\\
  0\ar[r] & M/\left(i(M') + j(M') \right)\ar[r, equal] & M/\left(i(M') + j(M') \right);
 \end{tikzcd}$\\

here, $\varphi:= \begin{pmatrix} \id & 0 \\ 0 & \id \\ 0 & 0 \end{pmatrix}$ and  $\psi:= \begin{pmatrix} 0 & 0 \\ \id & 0 \\ 0 & \id \end{pmatrix}$,
while the morphisms in the right-hand column are induced by the other morphisms. Hence, $x$ is equal to the difference between the classes of the top row and of the right-hand column. The objects in the top row are semi-simple, hence its class belongs to the image of $c$. The class of the right-hand column belongs to the image of $c$ by the case $\mathrm{length}(M') =1$ considered above. Hence so does $x$, as claimed.
\end{enumerate}
\item If $\mathrm{length}(U)=0$, we again use \cref{lem: images direct} and induction on $\mathrm{length}(M)$.
\end{enumerate}
\item We finally consider the case $\mathrm{length}(M') = 3$. Then $U:= i(M') \cap j(M')$ is of length 3, 2, 1 or 0.
\begin{enumerate}[leftmargin=1.8em]
\item If it is 3, we again apply \cref{lem: equal image} and \cref{lem:binary isomorphism}.
\item If it is 2, then $U' :=i^{-1}(U) \cap j^{-1}(U)$ is of length 2 or 1, but not of length 0 because then the object $M'$ of length 3 would contain the object $i^{-1}(U) \oplus j^{-1}(U)$ of length 4.
\begin{enumerate}[leftmargin=1.8em]
\item The case $\mathrm{length}(U')=2$ again follows from \cref{lem: equal image}, \cref{lem:binary isomorphism} and the case $\mathrm{length}(M') =1$ considered above.
\item If $\mathrm{length}(U') =1$, then $i(U') = j(U')$ or $i(U') \cap j(U') =0$.
\begin{enumerate}[label=(\Roman*), ref=\Roman*]
\item The case $i(U') = j(U')$  follows from \cref{lem: equal image}, \cref{lem:binary isomorphism} and the case $\mathrm{length}(M') =2$ considered above.
\item The case $i(U')\cap j(U') =0$ is dealt with as follows.
We have $U= i(M') \cap j(M') = i(U') \oplus j(U')$ as both sides are of length 2 and the right-hand side is contained in the left-hand side (indeed, for example, we have $i(U') =U \cap ij^{-1}(U) \subset U \subset j(M')$). This implies\\

$\begin{aligned}i^{-1}(U) = U' \oplus i^{-1}j(U') \qquad \textrm{ and } \qquad j^{-1}(U) = j^{-1}i(U') \oplus U'.\end{aligned}$\\

This in turn implies $M' = j^{-1} i (U') \oplus U' \oplus i^{-1}j (U')$. Note that the sum on the right-hand side here is indeed direct because otherwise we would have $i^{-1}j(U') \subset j^{-1}(U)$ and then that $i^{-1}(U) = U' \oplus i^{-1} j(U')  \subset j^{-1} (U)$ and finally that $U' = i^{-1} (U) \cap j^{-1}(U) = i^{-1}(U)$ would be of length 2. This in turn implies\\

$i(M') + j(M') = ij^{-1}i(U') \oplus i(U') \oplus j(U') \oplus ji^{-1}j(U').$\\

Again, note that the sum on the right-hand side here is indeed direct because otherwise we would have $ji^{-1}j(U') \subset i(M')$ and then $j(M') = i(U') \oplus j(U') \oplus ji^{-1}j(U') \subset i(M')$ and finally that $U = i(M') \cap j(M') = j(M')$ would be of length 3. We therefore obtain the following commutative diagram with acyclic rows and columns:\\

$\begin{tikzcd}
   U'^3\ar[d, "(j^{-1}i \, \id \, i^{-1}j)"]\ar[r, shift left, "\varphi"] \ar[r, shift right, "\psi"'] & U'^4 \ar[d, "(ij^{-1}i \; i\; j\; ji^{-1}j)"] \ar[r, shift left, "(0\, 0 \, 0 \id)"]\ar[r, shift right, "(\id 0\, 0\, 0)"']& U'\ar[d, shift right]\ar[d, shift left] \\[1em]
  M'\ar[d]\ar[r, shift left, "i"]\ar[r, shift right, "j"'] & M \ar[d] \ar[r, shift left, "p"]\ar[r, shift right, "q"']& M''\ar[d, shift right]\ar[d, shift left]\\ [1em]
  0\ar[r] & M/\left(i(M') + j(M') \right)\ar[r, equal] & M/\left(i(M') + j(M') \right);
 \end{tikzcd}$\\

here, $\varphi:= \begin{pmatrix} \id & 0 & 0\\ 0 & \id &0 \\ 0 & 0 & \id \\ 0 & 0 & 0 \end{pmatrix}$ and $\psi:= \begin{pmatrix} 0 & 0 & 0 \\ \id & 0 & 0 \\ 0 & \id & 0 \\ 0 & 0 & \id\end{pmatrix}$,
while the morphisms in the right-hand column are induced by the other morphisms. Hence $x$ is equal to the difference between the classes of the top row and of the right-hand column. The objects in the top row are semi-simple, hence its class belongs to the image of~$c$. The class of the right-hand column belongs to the image of $c$ by the case $\mathrm{length}(M') =1$ considered above. Hence $x$ belongs to the image of $c$, as claimed.
\end{enumerate}
\end{enumerate}
\item If $\mathrm{length}(U)=1$, then $U':= i^{-1}(U) \cap j^{-1}(U)$ is of length 1 or 0.
\begin{enumerate}[leftmargin=1em]
\item The case $\mathrm{length}(U') =1$ again follows from \cref{lem: equal image}, \cref{lem:binary isomorphism} and the case $\mathrm{length}(M') =2$ considered above.
\item The case $\mathrm{length}(U') = 0$ has been excluded in \cref{prop: length at most 3}.
\end{enumerate}
\item The final case $\mathrm{length}(U)=0$ has been excluded in \cref{prop: length at most 3}.
\end{enumerate}
\end{enumerate}
\end{proof}

\begin{cor}\label{cor: length M at most 6}
Suppose that the inclusion $\cB \subseteq \cA$ is as in (I). Then the element~$x$ belongs to the image of $c$ if $\mathrm{length}(M) \le 6$ except possibly when $\mathrm{length}(M') =\mathrm{length}(M'') = 3$, $U:= i(M') \cap j(M')$ is simple ($\iff M/(i(M') + j(M'))$ is simple) and both $i^{-1}(U) \cap j^{-1}(U)$ and $M''/(pj(M') + qi(M'))$ vanish.
\end{cor}

\begin{proof}
This follows from \cref{prop: length at most 3}. Note that the case $U =0$, which has been excluded in \cref{prop: length at most 3} as well, here follows from \cref{lem: images direct} and \cref{lem:binary isomorphism}.
\end{proof}

Some arguments in the proof of \cref{prop: length at most 3} appear more than once. This fosters the hope that the surjectivity of $c$ can be proved by suitable (nested) inductions. However, the case considered in the next example seems to require a new type of argument.  In particular, the arguments used so far don't yet suffice to give a complete inductive proof of the surjectivity of $c$.

The following example describes kind of the smallest short binary acyclic sequence $\begin{tikzcd}M'\ar[r, shift left, "i"]\ar[r, shift right, "j"'] & M \ar[r, shift left, "p"]\ar[r, shift right, "q"']& M''\end{tikzcd}$ for which the arguments used so far don't suffice to show that its class $x$ belongs to the image of $c$. 

\begin{ex}\label{ex: binary short exact sequence}
We write $C_n$ for the cyclic group $\bbZ/ n \bbZ$ of order $n$ and consider the short binary acyclic sequence
\[\bbC: \quad \begin{tikzcd} C_2 \oplus C_4 \ar[r, shift left, "i"] \ar[r, shift right, "j"'] & C_2 \oplus C_8 \oplus C_4 \ar[r, shift left, "p"] \ar[r, shift right, "q"']& C_2 \oplus C_4 \end{tikzcd}\]
in the abelian category $\cA$ of finitely generated $\bbZ/8\bbZ$-modules where
\[i(\bar{a},\bar{b}) := (\bar{a},\overline{2b},0), \quad j(\bar{a},\bar{b}) := (0, \overline{4a}, \bar{b}), \quad p(\bar{a},\bar{b},\bar{c}) := (\bar{b},\bar{c}), \quad q(\bar{a},\bar{b},\bar{c}) := (\bar{a},\bar{b}).\]
Then $\bbC$ satisfies the exceptional conditions in \cref{cor: length M at most 6}. Quillen's D\'evissage Theorem implies that $K_1(\cA)$ = $K_1(\bbF_2) = \bbF_2^\times$ is trivial. Hence the class of $\bbC$ in $K_1(\cA)$ is the sum of some Nenashev relations and some negative Nenashev relations. What are these relations? Surprisingly, there seems to not exist any $(3 \times 3)$-diagram in $\cA$ as in \cref{def:NenashevK1}(2) whose middle row is $\bbC$ and whose top and bottom row are non-zero; so, $\bbC$ probably has to be first related to some bigger binary complex before the relation in \cref{def:NenashevK1}(2) can be used. Finding this bigger complex looks to be quite an intractable problem.  
\end{ex}

\end{document}